\documentclass{article}

\usepackage{amsfonts}
\usepackage{amsthm}
\usepackage{amsmath}
\usepackage{amssymb}
\usepackage[all]{xy}
\usepackage{epic,eepic}
\usepackage{bm}

\newcommand{\C}{\mathbb C}

\newcommand{\N}{\mathbb N}

\newcommand{\Z}{\mathbb Z}

\newcommand{\eps}{\varepsilon}
\renewcommand{\-}{\setminus}
\newcommand{\norm}[1]{\left\Vert #1\right\Vert}

\newcommand{\abs}[1]{\left\vert #1 \right\vert}
\newcommand{\seq}[2]{(#1_{#2})_{#2\in\N}}

\newcommand{\proten}{\hat\otimes}
\newcommand{\ball}{\mathrm{ball}}
\newcommand{\hull}{\mathrm{co}}
\newcommand{\Proten}[3]{\hbox{$\widehat\bigotimes_{#1=1}^#2#3_#1$}}
\newcommand{\func}[5]
{\begin{eqnarray*} 
#1&:&#2\rightarrow#3,\\
&&#4\mapsto #5
\end{eqnarray*}}

\newcommand{\up}{\mathrm}

\newtheorem{thm}{Theorem}[section]
\newtheorem{cor}[thm]{Corollary}

\newtheorem{lem}[thm]{Lemma}
\newtheorem{prop}[thm]{Proposition}
\newtheorem{theorem}[thm]{Theorem}
\newtheorem{corollary}[thm]{Corollary}
\newtheorem{lemma}[thm]{Lemma}

\newtheorem{question}[thm]{Question}

\theoremstyle{definition}
\newtheorem{dfn}[thm]{Definition}
\theoremstyle{definition}
\newtheorem{definition}[thm]{Definition}
\theoremstyle{definition}
\newtheorem{Definition}[thm]{Definition}
\newtheorem{ex}[thm]{Example}

\newtheorem{example}[thm]{Example}
\title{Compact failure of multiplicativity for linear maps between Banach algebras.}
\author{M. J. Heath}
\begin{document}

\maketitle

The author is grateful for support from a PhD grant from the Engineering and Physical Sciences Research council of the UK and 
 post-doctoral fellowship number SFRH/BPD/40762/2007 from the Funda\c{c}\~{a}o para a Ci\'encia e a Tecnologia of Portugal. This paper contains some work included in the authors PhD thesis (\cite{MeThesis}).

\begin{abstract}
We introduce notions of compactness and weak compactness for multilinear maps from a product of normed spaces to a normed space, and prove some general results about these notions. We then consider linear maps $T:A\rightarrow B$ between Banach algebras that are ``close to multiplicative'' in the following senses: the failure of multiplicativity, defined by $S_T(a,b)=T(a)T(b)-T(ab)$, is  compact  [respectively weakly compact]. We call such maps cf-homomorphisms [respectively wcf-homomorphisms]. We also introduce a number of other, related definitions. We state and prove some general theorems about these maps when they are bounded, showing that they form categories and are closed under inversion of mappings and we give a variety of examples. We then turn our attention to commutative $C^*$-algebras and show that the behaviour of the various types of ``close-to-multiplicative'' maps depends on the existence of isolated points. Finally, we look at the splitting of Banach extensions when considered in the category of Banach algebras with bounded cf-homomorphisms [respectively wcf-homomorphisms] as the arrows. This relates to the (weak) compactness of 2-cocycles in the Hochschild-Kamowitz cohomology complex. We prove ``compact'' analogues of a number of established results in the Hochschild-Kamowitz cohomology theory.
\end{abstract}

\section{Introduction}
For Banach algebras, $A$ and $B$, and a linear map $T:A\rightarrow B$ we call the bilinear map 
$S_T:A\times A \rightarrow B$ given by $S_T(a,b)=T(a)T(b)-T(ab)$ the \emph{failure of multiplicativity of $T$}. Our focus in this 
paper will be new notions of  ``smallness'' for the failure of multiplicativity of a bounded linear map $T$, based on compactness. We will use these notions to produce new results in function theory and in the theory of Hochschild-Kamowitz cohomology of Banach algebras.    In order to provide some motivation we will first give a brief discussion of a well-studied notion of smallness for the failure of multiplicativity: namely that the norm of $S_T$ be less than some small $\delta>0$.  

 Let $A$ and $B$ be Banach algebras and let $\delta>0$. We call a bounded linear map $T:A\rightarrow B$  \emph{$\delta$-multiplicative} if $\norm{S_T}<\delta$. Now let $M(A,B)$ be the set of all multiplicative, bounded linear maps from $A$ to $B$ and, for $T\in \mathcal B(A,B)$, let $d(T)=\inf\{\norm{T-S}:S\in M(A,B)\}$.
In \cite{JohnsonAMNM}, Johnson proved the following.
\begin{prop}
Let $A$ and $B$ be Banach algebras and let $T:A\rightarrow B$ be a bounded linear map. Then
\[
\norm{S_T}\le(1+d(T)+2\norm T)d(T).
\]
\end{prop}
This implies that for every $\delta>0$ there exists $\eps>0$, such that all linear maps $T:A\rightarrow B$ with $\norm T<1$ that are 
within distance $\eps$ of some multiplicative bounded linear map are $\delta$-multiplicative. 
Research on $\delta$-multiplicativity has focused on when the converse of this holds.  
 We call $(A,B)$  an \emph{AMNM pair} (an ``almost multiplicative bounded linear maps are near multiplicative bounded linear maps 
pair'') if for every $\eps>0$ there exists $\delta>0$ such that all $\delta$-multiplicative bounded 
linear maps $T:A\rightarrow B$ with $\norm T<1$ are within distance $\eps$ of some multiplicative bounded linear map. If $(A,\C)$ is 
an AMNM-pair we call $A$ \emph{AMNM}. AMNM algebras are studied in \cite{JohnsonAMNMf}, \cite{Sidney-AM}, \cite{JaroszAM}, 
\cite{Sanchez} and \cite{Howey}, and a good source for AMNM pairs is \cite{JohnsonAMNM}. 
\subsection{Notation}
Let $X$ be a topological space and $S\subset X$. We write $\overline S$ for the closure of $S$. In the case that $X$ is a normed space then the closure is taken in the norm topology unless specifically stated otherwise.

For a normed space $E$ we write $\ball(E)$ for the open unit ball of $E$ and $\overline\ball (E)$ for the closed unit ball of $E$ (i.e. $\overline\ball (E)=\overline{\ball (E)}$).

The notions of smallness on which we shall concentrate in this paper are based on concepts of compactness for multilinear maps which we shall define and discuss in the following section.

\section{Compactness of multilinear maps.}
We start with the following definition, which will be important in this paper. In the special case of the norm topology the same condition is considered by Krikorian in \cite{Krikorian}. We made the more general version below independently. 
\begin{Definition}
Let $E_1,\dots,E_n$ be normed spaces and $F$ a vector space with a topology $\mathcal T$ defined on it. An $n$-linear map $T:E_1\times\dots\times E_n\rightarrow F$ is 
\emph{compact with respect to} $\mathcal T$ if  the closure in $\mathcal T$ of
\[
T(\{(x_1,\dots, x_n): x_i\in\ball(E_i)\})
\] 
is compact when considered as topological space with the subspace topology induced by  $\mathcal T$. 
If $\mathcal T$ is the norm topology we call $T$ \emph{compact}. If $\mathcal T$ is the weak topology we call $T$ \emph{weakly compact}. Let $E_1,\dots, E_n$ be normed spaces. We denote the set of all compact 
$n$-linear maps from $E_1\times\dots\times E_n$ to $F$ by $\mathcal K^n(E_1,\dots, E_n;F)$; in the case where $E_1=\dots=E_n=E$  we 
denote $\mathcal K^n(E_1,\dots, E_n;F)$ by $\mathcal K^n(E,F)$. We denote the set of all weakly-compact 
$n$-linear maps from $E_1\times\dots\times E_n$ to $F$ by $w\mathcal K^n(E_1,\dots, E_n;F)$; in the case where $E_1=\dots=E_n=E$ we   
denote $w\mathcal K^n(E_1,\dots, E_n;F)$ by $w\mathcal K^n(E,F)$.
\end{Definition}

We mention the following interesting source of examples of compact multilinear maps, due to Krikorian (\cite{Krikorian}).
\begin{ex}
Let $E$ and $F$ be Banach spaces, $U$ be an open subset of $E$, $n\in\N$,  and $f:U\rightarrow F$ be an $n$-times-continuously-differentiable function that maps bounded sets to relatively compact sets. Then, for $x\in U$ and $k\in \{1,\dots,n\}$ the $k$th derivative of $f$ at $x$ is a compact $k$-linear map from $E\times\dots\times E$ to $F$. 
\end{ex}
We shall now prove that $T$ being (weakly) compact is equivalent to the associated linear map from the $n$-fold projective tensor 
product of $E$ with itself to $F$ being (weakly) compact. We refer the reader unfamiliar with this construction to 
\cite[Appendix A1]{Dales} for definitions and notation. We shall need the following result; with $n=2$ it is \cite[A.3.69]{Dales} and the general version is similar.
\begin{prop}\label{ptp}
Let $E_1,\dots E_n$ be normed spaces and $F$ be a Banach space, and let $T\in\mathcal B(E_1,\dots, E_n; F)$. Then there is a unique, 
bounded, linear map $\widetilde T:\Proten{i}{n}{E}\rightarrow F$ such that 
$$\widetilde T(x_1\proten\dots\proten x_n)=T(x_1,\dots, x_n)\quad(x_j\in E_j, j\in\{1,\dots,n\}).$$
Furthermore, the map $T\mapsto \widetilde T$, $\mathcal B^n(E_1,\dots,E_n;F)\rightarrow\mathcal B\left(\Proten{i}{n}{E},F\right)$, is an 
isometric Banach space isomorphism.
\end{prop} The following is elementary.
\begin{lemma}\label{convex}
Let $E_1,\dots,E_n$ be normed spaces. Then $\overline\ball\left(\widehat\bigotimes_{i=1}^nE_i\right)$ is the closed convex hull of $\{x_1\hat\otimes\dots\hat\otimes x_n: x_i\in\ball(E_i)\}$.
\end{lemma}
%
The following is a special case of \cite[Theorem IV.5]{Bourbaki} (Krein's Theorem).
\begin{prop}\label{Bour}
Let $F$ be a Banach space and $\mathcal T$ a topology on $F$ for which a functional on $F$ is $\mathcal T$-continuous if and only if
it is norm continuous (equivalently $\mathcal T$ contains the weak topology and is contained in the norm topology). Let $S$ be a subset of of $F$ which is compact with respect to $\mathcal T$. Then the closed convex balanced hull (and 
therefore the closed convex hull) of $S$ is compact with respect to  $\mathcal T$.
\end{prop}
In particular if $\mathcal T$ is the norm topology or the weak topology the above holds.

In the special case where $n=2$ and $\mathcal T$ is the norm topology, the following theorem and corollary were proven in \cite{Krikorian}. We proved this more general version 
independently.
\begin{theorem}\label{compml}
Let $F$ and $\mathcal T$ be as in Proposition \ref{Bour} and let $E_1,\dots,E_n$ be normed spaces. A multilinear map 
$T\in \mathcal B(E_1, \dots, E_n;F)$ is compact with respect to $\mathcal T$ if 
and only if the linear map $\widetilde T \in \mathcal B\left(\widehat\bigotimes_{i=1}^nE_i, F\right)$ with 
$\widetilde T(x_1\hat\otimes\dots\hat\otimes x_n)=T(x_1,\dots, x_n)$ is compact with respect to $\mathcal T$.
\end{theorem}
\begin{proof}
In this proof ``$\mathcal T$-compact'' will mean ``compact with respect to $\mathcal T$''.  First assume that $T$ is 
$\mathcal T$-compact. Then 
$$\overline{T(\{(x_1,\dots, x_n): x_i\in\ball(E_i)\})}=
\overline{\widetilde T(\{x_1\hat\otimes\dots\hat\otimes x_n: x_i\in\ball(E_i)\})}$$ 
is $\mathcal T$-compact. We call this set $S$. By Proposition \ref{Bour}, $\overline\hull(S)$ is also $\mathcal T$-compact.
\begin{eqnarray*}
K:&=&\overline{\widetilde T(\overline\hull(\{x_1\hat\otimes\dots\hat\otimes x_n: x_i\in\ball(E_i)\}}))\\
&=&\overline{\widetilde T(\hull(\{x_1\hat\otimes\dots\hat\otimes x_n: x_i\in\ball(E_i)\}}))\\
&=&\overline{\hull(\widetilde T(\{x_1\hat\otimes\dots\hat\otimes x_n: x_i\in\ball(E_i)\}))}\subseteq\overline\hull(S),
\end{eqnarray*}
and so the $K$ is $\mathcal T$-compact. By Lemma \ref{convex}, we have 
$K=\overline{\widetilde T\left(\overline\ball(\widehat\bigotimes_{i=1}^nE_i)\right)}$ and 
hence $\widetilde T$ is $\mathcal T$-compact.
\paragraph{}We now assume that $\widetilde T$ is compact. We have,
\begin{eqnarray*}
 \overline{T\big(\{(x_1,\dots,x_n):x_i\in\ball(E_i)\}\big)}
&=&\overline{\widetilde T\big(\{(x_1\proten\dots\proten x_n):x_i\in\ball(E_i)\}\big)}\\
&\subseteq& \overline{\widetilde T\big(\{z\in\Proten{i}{n}{E}:\norm z<1\}\big)}.
\end{eqnarray*}
The right-hand side of the above expression is compact and hence so is the left hand side; the result follows.
\end{proof}
In the following corollary and its proof ``closed'' means ``closed with respect to the norm topology''.
\begin{cor}\label{comvec}
Let $E_1,\dots, E_n$ and $F$ be Banach spaces; then $\mathcal K^n(E_1,\dots, E_n;F)$ and $w\mathcal K^n(E_1,\dots, E_n;F)$ are 
closed subspaces of $\mathcal B^n(E_1,\dots, E_n;F)$ such that $\mathcal K^n(E_1,\dots, E_n;F)$ is isometrically isomorphic to 
$\mathcal K\left(\Proten{i}{n}{E},F\right)$ and  \newline $w\mathcal K^n(E_1,\dots, E_n;F)$ is isometrically isomorphic to 
$w\mathcal K\left(\Proten{i}{n}{E},F\right)$.
\end{cor}
\begin{proof}
Let $E$ and $F$ be Banach spaces. Then it is standard that $\mathcal K(E,F)$ and $w\mathcal K(E,F)$ are closed subspaces of 
$\mathcal B(E,F)$. Thus $\mathcal K\left(\Proten{i}{n}{E},F\right)$ and $w\mathcal K\left(\Proten{i}{n}{E},F\right)$ are Banach 
spaces. By Proposition \ref{ptp} and Theorem \ref{compml},  $\mathcal K^n(E_1,\dots, E_n;F)$ is isometrically isomorphic to 
$\mathcal K\left(\Proten{i}{n}{E},F\right)$ and \newline $w\mathcal K^n(E_1,\dots, E_n;F)$ is isometrically isomorphic to 
$w\mathcal K\left(\Proten{i}{n}{E},F\right)$. Hence, they are complete, and therefore closed, subspaces of 
$\mathcal B^n(E_1,\dots, E_n;F)$ and the result follows.
\end{proof}
\subsection{(Weakly) compact failure of multiplicativity for linear maps.}
In this subsection we introduce a class of functions between Banach algebras for which are, in a certain sense, ``close'' to being 
homomorphisms. 
\begin{dfn} \label{nhdfn}
Let $A$ and $B$ be normed algebras and let $T$ be a linear map. Recall the definition of the failure of multiplicativity $S_T$.
We call a $T$, a \emph{cf-homomorphism} (where  ``cf'' stands for ``compact from'') if $S_T$ is compact and a \emph{wcf-homomorphism} (where  ``wcf'' stands for ``weakly compact from'') if $S_T$ is weakly compact. If $S_T$ is finite-dimensional we call $T$ an \emph{fdf-homomorphism} and if for $n\in \N$ is at most $n$-dimensional we call $T$ an \emph{$n$df-homomorphism}.

If, for each, $a\in A$ we have that $S_T(f,\cdot)$ and $S_T(\cdot, f)$ are compact linear maps, we say that $T$ is a \emph{semi-cf-homomorphism}. We define ``\emph{semi-wcf-homomorphism}'', ``\emph{semi-fdf-homomorphism}'' and ``\emph{semi-$n$df-homomorphism}'' similarly.
\end{dfn}

We note some obvious relationships between these conditions. The set of cf-homomorphisms from $A$ to $B$ contains all homomorphisms from $A$ to  $B$ and all compact linear maps 
from $A$ to $B$.  The set of wcf-homomorphisms from $A$ to $B$ contains all weakly compact linear maps from $A$ to $B$. Adding ``w'' or ``semi-'' to a condition makes it weaker. An $n$df-homomorphism is an $(n+1)$df-homomorphism and a continuous fdf-homomorphism is a cf-homomorphism. 

Let $E$ and $F$ be Banach spaces. We say a map $T:E\rightarrow F$ is \emph{weak-weak} continuous, if it is continuous when considered as a map from $E$ equipped with the weak topology to $F$ equipped with the weak topology. The following is part of \cite[27.6]{Jameson}.
\begin{prop}\label{w-w}
Let $E$ and $F$ be Banach spaces and $T:E\rightarrow F$ a bounded linear map. Then $T$ is weak-weak continuous.
\end{prop}
We shall need the following lemma.
\begin{lem}\label{comcom}
Let $E$, $F$ and $G$ be Banach spaces, $T_1\in\mathcal B(G,E),T_2\in\mathcal B(F,G)$ and $S\in\mathcal B^2(E,F)$. We define a map 
$R\in\mathcal B^2(G,F)$ by $R(g_1,g_2)=S(T_1(g_1),T_2(g_1))$. Suppose that $S\in \mathcal K^2(E,F)$. Then $R\in\mathcal K^2(G,F)$ and 
$T_2\circ S\in\mathcal K^2(E,G)$. Suppose instead that $S\in w\mathcal K^2(E,F)$. Then $R\in w\mathcal K^2(G,F)$ and 
$T_2\circ S\in w\mathcal K ^2(E,G)$.
\end{lem}
\begin{proof}
In this proof we shall write ``$X^{(2)}$'' for the Cartesian product of a set $X$ with itself. Closures will initially be in the norm topology. Suppose $S\in \mathcal K^2(E,F)$.
\begin{eqnarray}
\nonumber\overline {R\left(\ball(G)^{(2)}\right)}&=&\overline{S(\{(T_1(g_1),T_1(g_2)):g_1,g_2\in\ball(G)\})}\\
\nonumber&\subseteq& \overline{S(\{(e_1,e_2):e_1,e_2\in\norm T_1\ball(E)\})}\\
&=&\norm {T_1}^2\overline{S\left(\ball(E)^{(2)}\right)},
\label{nheq3}
\end{eqnarray}
but $\overline{S\left(\ball(E)^{(2)}\right)}$ is compact and so $\overline {R(\ball(G)^{(2)})}$ is a closed subset of the 
compact set $\norm {T_1}^2\overline{S\left(\ball(E)^{(2)}\right)}$ and thus is compact. Hence $R\in\mathcal K^2(G,F)$.
Also, 
\begin{equation}\label{nheq1}
 \overline {T_2\circ S(\ball(G)^{(2)})}=\overline{T_2(S(\ball(E)^{(2)}))}=\overline{T_2\left(\overline{S(\ball(E)^{(2)})}\right)},
\end{equation}
but $\overline{S(\ball(E)^{(2)})}$ is compact, and so (since $T_2$ is bounded)
\begin{equation}\label{nheq2}
 \overline {S(\ball(E)^{(2)})}=\overline{T_2\left(\overline{S(\ball(E)^{(2)})}\right)}
=T_2\left(\overline{S(\ball(E)^{(2)})}\right),
\end{equation}
is compact. Hence, $T_2\circ S\in\mathcal K^2(E,G)$.

Now suppose $S\in w\mathcal K^2(E,F)$. By Proposition \ref{w-w}, $T_1$ and $T_2$ are continuous when $E$, $F$ and $G$ are 
considered with the weak topology. Hence, each of (\ref{nheq1}), (\ref{nheq2}) and (\ref{nheq3}) holds with the closure taken in the weak 
topology and so the result follows as in the norm topology case.
\end{proof}
\begin{thm}\label{nhcomp}
Let $A,B$ and $C$ be Banach algebras and let $T_1:A\rightarrow B$ and $T_2:B\rightarrow C$ be bounded cf-homomorphisms 
[respectively bounded wcf-homomorphisms]. Then $T_2\circ T_1$ is a bounded cf-homomorphism [respectively bounded  wcf-homomorphism].
\end{thm}
\begin{proof}
Let $a,b\in A$. Then, a direct calculation gives,
\[
S_{T_2\circ T_1}(a,b)=T_2( S_{T_1}(a,b))-S_{T_2}(T_1(a),T_1(b)).
\]
Thence, the result is immediate from Lemma \ref{comcom} and Corollary \ref{comvec}. 
\end{proof}
Hence we have that the class of Banach algebras together with bounded cf-homomorphisms and the class of Banach algebras together with bounded wcf-homomorphism form concrete categories. 
\begin{thm}\label{coninv}
Let $A$ and $B$ be Banach algebras and $T$ be a bounded cf-homomorphism [respectively a bounded wcf-homomorphism] that is bijective. Then the 
inverse mapping $T^{-1}:B\rightarrow A$ is a bounded cf-homomorphism [respectively a bounded wcf-homomorphism].
\end{thm}
\begin{proof}
 By the Banach isomorphism theorem,  $T^{-1}$ is a bounded linear map. Also, for $b,b'\in B$ a direct calculation yields,
\[
S_{T^{-1}}(b,b')=-T^{-1}\circ S_T(T^{-1}(b),T^{-1}(b')).
\]
Thus, the result follows from Lemma \ref{comcom}.
\end{proof}
Hence, a morphism in the category of Banach algebras with bounded cf-homomorphisms or in the category of Banach algebras with 
bounded wcf-\\homomorphisms is an isomorphism if and only if it is bijective.
\section{Some examples}

We give an example to show that bounded semi-cf-isomorphisms need not be wcf-homomorphisms.
\begin{ex}
 Following \cite{LRRW} we say a Banach algebra, $A$ has \emph{compact multiplication} if, for each $a\in A$, left and right 
multiplication by $a$ -- $L_a$ and $R_a$ -- are both compact operators on $A$. Let $A$ and $B$ be Banach algebras with compact multiplication and 
let $T:A\rightarrow B$ be a bounded linear map. Then it is clear that $T$ is automatically a semi-cf-homomorphism. Now let $A=c_0$, with pointwise multiplication. Then $A$ has compact multiplication since if we take $(a_n)\subset c_{00}$ such that $a_n\rightarrow a$ then $L_{a_n}\rightarrow L_a=R_a$ and each $L_{a_n}$ has finite rank. Let  $T:A\rightarrow A$ be given by $T(a)=2a$; then $S_T$ is surjective. Since the unit ball of $c_0$ is not relatively compact in the weak topology, it follows that $T$ is not a wcf-homomorphism.
\end{ex}

We now show that  weakly compact linear maps need not be semi-cf-\\homomorphisms.
\begin{ex}
Let $A$ be an infinite-dimensional, unital Banach algebra, which is reflexive as a Banach space. For example let $A_0$ be $\ell^2$ with pointwise multiplication and let $A$ be the one-dimensional unitisation of $A_0$. Denote the unit of $A$ by $e$. Now 
let $T:A\rightarrow A$  be given by $T(a)=2a$ for each $a\in A$. Then $S_T(e,\cdot)=T$ which is 
bijective, and so not compact. 
\end{ex}
  We now give examples to show that two Banach algebras 
may be isomorphic in the category of Banach algebras with bounded cf-homomorphisms without being isomorphic in the usual category of Banach 
algebras.
\begin{ex}
Let $n\in\N$ and let $A$ and $B$ be non-isomorphic Banach algebras, each with underlying vector space $\C^n$. Clearly, 
the identity map from $\C^n$ to itself defines a compact linear isomorphism from $A$ to $B$ and so  it defines an isomorphism in the category of Banach algebras with bounded cf-homomorphisms.  In particular we may take $A$ to be $\C$ with the usual product and $B$ to be $\C$ with zero product. If we let $A=C(\{1,2,3,4\})$ and $B=\mathcal B(\ell^2(\{1,2\}))$ (the algebra of $2\times 2$ matrices over $\C$)  we have an example where both algebras are unital $C^*$-algebras.
\end{ex}
We now give an example of infinite-dimensional Banach algebras which are  isomorphic in the category of Banach algebras with bounded cf-homomorphisms without being isomorphic in the usual category of Banach 
algebras.
\begin{ex}
Let $A=\ell^\infty$ with the pointwise product and let 
$B$ be the vector space  $\mathcal B(\ell^2(\{1,2\}))\oplus A$ with the norm $\norm{(\mathbf A,a)}=\max\{\norm{\mathbf A},\norm a\}$ 
and the product $(\mathbf A,a)(\mathbf B,b)=(\mathbf{AB},ab)$ $(\mathbf A\in \mathcal B(\ell^2(\{1,2\}), a\in A)$. Then $B$ is a 
Banach algebra.
We define a linear map $T:A\rightarrow B$ by 
\[
(a_n)_{n\in\N}\rightarrow \left(\left[\begin{array}{ll}a_1,&a_2\\a_3,&a_4\end{array}\right],(a_{n-4})_{n\in\N}\right).
\]
It is easy to check that $T$ is a bounded linear isomorphism. Also 
\[
S_T\left((a_n)_{n\in\N},(b_n)_{n\in\N}\right)=\left(\left[\begin{array}{ll}a_1,&a_2\\a_3,&a_4\end{array}\right]
\left[\begin{array}{ll}b_1,&b_2\\b_3,&b_4\end{array}\right]-\left[\begin{array}{ll}a_1b_1,&a_2b_2\\a_3b_3,&a_4b_4\end{array}\right],
0\right)
\]
and so $S_T$ is of finite rank. Thus $T$ is an isomorphism in the category of Banach algebras and cf-homomorphisms. Clearly, $A$ is commutative and $B$ is not so they are not isomorphic as Banach algebras.
\end{ex}
We now give an example of a bounded $1$df-homomorphism (and hence of a cf-homomorphism) that is neither a homomorphism nor a compact linear map. 
\begin{ex}\label{nhex}
Let $A$ be $\ell^\infty$ with pointwise multiplication. We let $e_k\in A$ be the sequence with $1$ in the $k$th place and $0$ in all 
other places. We define a bounded linear map $T:A\rightarrow A$ by $T(a)=a+a_1e_1$, $(a=\seq{a}{k}\in A)$. Then $T$ is a linear isomorphism from $\ell^\infty$ to itself and, hence, is not weakly-compact. Also, $e_1^2=e_1$ and $T(e_1)=2e_1$, so
\[
T(2e_1e_1)=T(2e_1)=4e_1\ne8e_1=T(2e_1)T(e_1).
\] 
Hence, $T$ is not a homomorphism.
However, for $a=\seq{a}{k},\,b=\seq{b}{k}\in A$,
\begin{eqnarray*}
S_T(a,b)&=&T(ab)-T(a)T(b)\\
&=&ab+a_1b_1e_1-(a+a_1e_1)(b+b_1e_1)\\
&=&ab+a_1b_1e_1-(ab+3a_1b_1e_1)=-2a_1b_1e_1.
\end{eqnarray*}
This has rank 1 and so is compact.
\end{ex}
Since cf-homomorphisms are ``a compact map away from being homomorphisms'' one may conjecture that if $A$ and $B$ are Banach 
algebras, $T_1:A\rightarrow B$ is a continuous homomorphism and $T_2:A\rightarrow B$ is a compact linear map then $T:=T_1+T_2$ must be a 
cf-homomorphism. The following example shows that this is not true, even if $T_2$ is a rank 1 homomorphism.
\begin{ex}
Let $A$ and $e_k\in A$ be as in Example \ref{nhex} and denote the identity element of $A$ by $1$. Let $T_1:A\rightarrow A$ be the identity homomorphism $T_1(a)=a$, 
$a\in A$ 
and let $T_2:A\rightarrow A$ be the bounded, rank 1, linear map given by $T_2(a)=a_11$. Then, if $T=T_1+T_2$, 
\[
T(e_k)=\left\{
\begin{array}{ll}
e_1+1&\textrm{if } k=1,\\
e_k&\textrm{otherwise}.
\end{array}.\right.
\]
Hence, for $k>1$, 
\[
S_T(e_1,e_k)=T(e_1)T(e_k)-T(e_1e_k)=e_k,
\]
so 
\[
(e_k)_{k=2}^\infty\subseteq S_T\left(\overline{\ball(A)}^{(2)}\right)\subseteq\overline{S_T\left(\ball(A)^{(2)}\right)},
\]
but $(e_k)_{k=2}^\infty$ has no convergent subsequence, so $T$ is not a cf-homomorphism.
\end{ex}
\section{Commutative $C^*$-algebras.}
In this section we discuss how these notions relate to functions on locally-compact, Hausdorff topological spaces. Many of the results could be extended to more general classes  of Banach function algebras, but to avoid having to give a large number of definitions we shall stick to the case of commutative $C^*$-algebras. 
\subsection{Nowhere-zero preserving maps}
The following is the Gleason-Kahane-\.Zelazko theorem, which may be found as \cite[Theorem 2.3]{Jarosz}. The author would like to thank Joel Feinstein for pointing him towards this result.
\begin{prop}\label{GKZ}
 Let $A$ be a Banach algebra with a unit denoted $1$ and let $\phi$ be a linear functional on $A$. Assume that for each invertible $f\in A$, $\phi(f)\ne0$. Then $\phi/\phi(1)$ is multiplicative. 
\end{prop}
This has the following corollary.
\begin{cor}\label{GKZ1}
 Let $A$ be a Banach algebra with a unit denoted $1$, $B$ be a  commutative, semisimple Banach algebra and 
$T:A\rightarrow B$ be a linear map.  Assume that for each invertible $f\in A$, $T(f)$ is invertible in $B$. Then $T/T(1)$ is multiplicative. 
\end{cor}
\begin{proof}
Via the Gel$'$fand transform, we may assume that $B$ is a Banach function algebra on its character space. Therefore, $b\in B$ is invertible if $\phi(b)\ne 0$ for each multiplicative linear functional $\phi$ on $B$, and $T$ is multiplicative if and only if $T\circ \phi$ is multiplicative for each multiplicative linear functional $\phi$ on $B$. Thus, the result follows from Proposition \ref{GKZ}.
\end{proof}
\begin{dfn}
 We say a topological space is \emph{perfect} if it is non-empty and has no isolated points.
\end{dfn}

\begin{thm}\label{perfect}
Let $X$ and $Y$ be infinite, locally-compact, Hausdorff spaces.
If  $X$ is perfect  and $T:C_0(X)\rightarrow C_0(Y)$ is a bijective bounded semi-wcf-morphism, then $T$ is multiplicative and thus is of the form $T(f)=f\circ \psi$ for some homeomorphism $\psi:Y\rightarrow X$.
\end{thm}
\begin{proof}
First, we reduce this to the unital case. Assume for now that the result holds in the case where $X$ and $Y$ are both compact (i.e. that $C_0(X)=C(X)$ and $C_0(Y)=C(Y)$ are unital). Let $X$ and $Y$ be locally-compact, Hausdorff spaces and let $T:C_0(X)\rightarrow C_0(Y)$ be a bounded semi-wcf-homomorphism. Let $\widetilde X$ and $\widetilde Y$ be the unconditional one-point compactification of $X$ and $Y$ respectively (that is the compact space obtained by adjoining an extra point whether or not the original space was compact). Let $f\in C_0(X)$. We define a map as follows:
\func{\widetilde T}{C(\widetilde X)}{C(\widetilde Y)}{f+\alpha1_{\widetilde X}}{T(f)+\alpha1_{\widetilde Y}.}

Now, let $g\in C_0(X)$  and $(f_i)_i\subset C_0(X)$  be a bounded net. Then $(S_{\tilde T}(f_i,g))_i=(S_T(f_i,g))_i$ has a weakly convergent subnet. By symmetry it follows that $T$ is a semi-wcf-homomorphism.  Hence, by assumption, $T$ is multiplicative. Now let $f,g\in C_0(X)$ and $\alpha, \beta \in \C$. We have 
\begin{eqnarray}
 \nonumber S_{\widetilde T}((f+\alpha1_{\widetilde X}),(g+\beta1_{\widetilde X}))&=&\\
\nonumber\widetilde T((f+\alpha1_{\widetilde X})(g+\beta1_{\widetilde X}))-\widetilde T(f+\alpha 1_X)\widetilde T(g+\beta 1_Y)&=&\\
 \nonumber\widetilde T(fg+\beta f+\alpha g +\alpha\beta 1_X)-\widetilde T(f)\widetilde T(g)-\widetilde T(\alpha g)-\widetilde T(\beta f) -\alpha\beta 1_X&=&\\
\nonumber\widetilde T(fg)-\widetilde T(f)\widetilde T(g)&=&\\
S_T(f,g)&=&0,
\end{eqnarray}
and so $T$ is multiplicative.

Henceforth we assume that $X$ and $Y$ are compact. Let $T:C_0(X)\rightarrow C_0(Y)$ be a bijective bounded semi-wcf-morphism.
We show that $T(1_X)=1_Y$. Assume otherwise; then 
\[
 S_T(1_X,\cdot)=T(1_X)T(\cdot)-T(\cdot)=(T(1_X)-1_Y)T(\cdot).
\]
Since $T$ is a bounded linear isomorphism, it follows that multiplication by $T(1_X)-1_Y$ is weakly compact. It is easy to see that this means that there is no infinite subset of $Y$ on which $T(1_X)-1_Y$ is not bounded below. Since $Y$ is perfect it follows that $T(1_X)=1_Y$.

We now assume, towards a contradiction, that $T$ is not multiplicative. Then $T^{-1}$ is also not multiplicative. By Corollary \ref{GKZ1} 
it follows that there exists a non-invertible $f\in C(X)$ such that $T(f)$ is invertible. Since the set of invertibles in $C(Y)$ is open, and since for any $x\in X$, 
\[\{f\in C(X):\textrm{there is a neighbourhood $U$ of $x$ with }f(U)=\{0\}\}\]
is dense in $\{f\in C(X):f(x)=0\}$, we may assume without loss of generality that there exists some non-empty, open subset $U\subset X$ such that 
$f(U)=\{0\}$. We can then take $(g_n)\subset C(X)$ such that, for each $n$, we have that $g_n(X\setminus U)\subseteq\{0\}$ and $(g_n)$ has no 
weakly convergent subsequence. Thus, we have that $fg_n=0$ and so
\[
S_T(f,g_n)=T(fg_n)-T(f)T(g_n)=-T(f)T(g_n).
\]
Since $T$ is a Banach space isomorphism and $T(f)$ is invertible, the map $g\mapsto -T(f)T(g_n)$ is a Banach space isomorphism. Thus $S_T(f,g_n)$ has no weakly 
convergent subsequence and so $S_T(f,\cdot)$ is not weakly compact, which contradicts our original assumption that $T$ is a semi-wcf-homomorphism.

It is standard that Banach algebra isomorphisms from $C_0(X)$ to $C_0(Y)$ are of the form $f\mapsto f\circ \psi$ for some homeomorphism $\psi: Y\rightarrow X$.
\end{proof}

For locally-compact Hausdorff spaces with isolated points the situation is quite different. Indeed, we have the following.
\begin{cor}
 Let $X$ be an infinite, locally-compact Hausdorff  space. Then the following are equivalent:
\begin{itemize}
 \item [(a)] every bounded semi-wcf-isomorphism from $C_0(X)$ to $C_0(X)$ is multiplicative;
\item[(b)] every bounded $1$df-isomorphism is nowhere-zero preserving;
\item[(c)] $X$ is perfect.
\end{itemize}
\end{cor}
\begin{proof}
 Clearly (a) implies (b) and it follows from Theorem \ref{perfect} that (c) implies (a). It remains to show that (b) implies 
(c). Assume that $X$ is not perfect, let $x_0\in X$ be an isolated point and $x_1\in X\-\{x_0\}$. Define $e_{x_0}\in C_0(X)$ by $e_{x_0}(x_0)=1$ and $e_{x_0}\left(X\-\{x_0\}\right)=\{0\}$.  Then we can define a Banach space isomorphism 
$T:C_0(X)\rightarrow C_0(X)$ by
\[
T(f)(x)=\left\{
\begin{array}{ll}
f(x_0)-f(x_1)&\textrm{if } x=x_0,\\
f(x)&\textrm{otherwise}.
\end{array}.\right.
\]
Now $S_T(A\times A)\in e_{x_0}\C$, and so $T$ is a $1$df-isomorphism. However, if $f\in C_0(X)$ is nowhere zero, then $g:=f+(f(x_1)-f(x_0))e_{x_0}\in A$ is also nowhere zero but has $T(g)(x_0)=0$. The result follows.
\end{proof}

\section{[W]cf-splitting of Banach extentions and compactness of Kamowitz cocycles}
These notions of compact failure of multiplicativity fit together nicely with theory of Banach extensions and Kamowitz's cohomology theory for Banach algebras (see \cite{Kam} and \cite{BDL}). Before we discuss this connection, we shall use the next two subsections to lay out some definitions and basic results we shall need from this field of study.
\subsection{Hochschild cohomology}\label{Hc}
Let $A$ be an algebra and $E$ an $A$-bimodule. For $x\in E$, we define $\delta_x:A\rightarrow E$ by 
$\delta_x(a)=a\cdot x-x\cdot a$ and define a linear map $\delta^0:E\rightarrow \mathcal L(A,E)$ by 
$x\mapsto \delta_x$. For $n\in\N$ and $T\in\mathcal L^n(A,E)$ we let
\begin{eqnarray*}
\delta^nT(a_1, \dots,a_{n+1})=&& a_1\cdot T(a_2,\dots, a_{n+1})+(-1)^{n+1}T(a_1\dots, a_n)\cdot a_{n+1}\\
&&+\sum_{j=1}^n(-1)^jT(a_1,\dots,a_{j-1},a_ja_{j+1},a_{j+2},\dots,a_{n+1}).
\end{eqnarray*}
It is standard (see for example \cite[p. 127]{Dales}) that this yields a complex,
\begin{eqnarray*}
 \mathcal L^\bullet(A,E):0&\rightarrow&E\stackrel{\delta^0}{\rightarrow}\mathcal L(A,E)\stackrel{\delta^1}{\rightarrow}
\mathcal L^2(A,E)\dots\\
&\stackrel{\delta^{n-1}}{\rightarrow}&\mathcal L^n(A,E)\stackrel{\delta^n}{\rightarrow}\mathcal L^{n+1}(A,E)\dots,
\end{eqnarray*}
of vector spaces and linear maps.
We define the following vector spaces:
$$ Z^n(A,E):= \textrm{ker }\delta^n,\quad  N^n(A,E):= \textrm{im }\delta^{n-1}.$$
We call the elements of $Z^n(A,E)$ \index{cocycle}\emph{$n$-cocycles} and the elements of $N^n$ \index{coboundary}\emph{$n$-coboundaries}.
\begin{dfn}Let $A$ be an algebra and $E$ an $A$-bimodule. We define the 
\emph{$n$th cohomology group of $A$ with coefficients in $E$} to be 
$$H^n(A,E):=Z^n(A,E)/N^n(A,E).$$ 
\end{dfn}
\begin{dfn}
Let $A$ and $I$ be algebras. An \index{extension}\emph{extension of $A$ by $I$} is a short exact sequence of algebras and algebra homomorphisms
\[
\Sigma=\Sigma(\mathfrak A;I):0\rightarrow I\stackrel{\iota}{\rightarrow}\mathfrak{A}\stackrel{q}{\rightarrow}A\rightarrow 0, 
\] 
where $\iota(I)$ is an ideal in $\mathfrak A$. We say $\Sigma$ is \index{extension!singular}\emph{singular} if $I^2=0$ (that is, $I$ has the zero 
multiplication). 

Two extensions, 
$\Sigma(\mathfrak A;I):0\rightarrow I\stackrel{\iota_{\mathfrak A}}{\rightarrow}\mathfrak{A}
\stackrel{q_{\mathfrak A}}{\rightarrow}A\rightarrow 0 $ 
and 
$\Sigma(\mathfrak B;I):0\rightarrow I\stackrel{\iota_{\mathfrak B}}{\rightarrow}\mathfrak{B}
\stackrel{q_{\mathfrak B}}{\rightarrow}A\rightarrow 0$ are \index{extension!s, equivalent}\emph{equivalent} if there is an algebra isomorphism 
$\psi:\mathfrak A\rightarrow\mathfrak B$ such that $\psi(x)=x$ ($x\in I$) and the diagram
\begin{equation}\label{comdia}
 \xymatrix{
0\ar[r]&I \ar[r]^{\iota_{\mathfrak A}}\ar@{<->}[d]_{\up{id}}&\mathfrak A\ar[r]^{q_\mathfrak A}\ar[d]_\psi&A \ar@{<->}[d]^{\up{id}}\ar[r] &0\\
0\ar[r]&I \ar[r]^{\iota_{\mathfrak B}}\ar[u]&\mathfrak B\ar[r]^{q_\mathfrak B}\ar[u]_{\psi^{-1}}&A \ar[r] &0}
\end{equation}
commutes. It is standard that this defines an equivalence relation on the class of all extensions of $A$ by $I$. We denote the equivalence class, with respect to 
equivalence, of an extension $\Sigma$ by $[\Sigma]$.  The extension $\Sigma$ \index{extension!splitting}\emph{splits} if there is a homomorphism $\theta:A\rightarrow \mathfrak A$ such that $q\circ\theta=\up{id_A}$; we call $\theta$ a \emph{splitting homomorphism for} $\Sigma$.
\end{dfn}
\subsection{Kamowitz cohomology} \label{Kamsub}\index{cohomology!Kamowitz}
Let $A$ be a Banach algebra and $E$ be a Banach $A$-bimodule. It is standard (see for example \cite[p.273]{Dales}) that for each 
$n\in\Z^+$, $\delta^n(\mathcal B^n(A,E))\subseteq \mathcal B^{n+1}(A,E)$. Furthermore, if we define, 
$\gamma^n:\mathcal B^n(A,E)\rightarrow \mathcal B^{n+1}(A,E)$ by $\gamma^n(T)=\delta^n(T)$,  then 
\begin{eqnarray*}
 \mathcal B^\bullet(A,E):0&\rightarrow&E\stackrel{\gamma^0}{\rightarrow}\mathcal B(A,E)\stackrel{\gamma^1}{\rightarrow}
\mathcal B^2(A,E)\dots\\
&\stackrel{\gamma^{n-1}}{\rightarrow}&\mathcal B^n(A,E)\stackrel{\gamma^n}{\rightarrow}\mathcal B^{n+1}(A,E)\dots
\end{eqnarray*}
is a complex of Banach spaces and continuous linear maps. We note that in the above complex it is more usual to call the maps ``$\delta^n$'' but we wanted to avoid any confusion with the purely algebraic case.
Thence we may define the following vector spaces:
\begin{eqnarray*}
\mathcal Z^n(A,E)&:=& \textrm{ker }\gamma^n,\quad \mathcal N^n(A,E):= \textrm{im }\gamma^{n-1};\\ 
\mathcal H^n(A,E)&:=&\mathcal Z^n(A,E)/\mathcal N^n(A,E);\\
\widetilde N^n(A,E)&:=& N^n(A,E)\cap\mathcal B(A,E)(=\mathcal Z^n(A,E)\cap N^n(A,E));\\
\widetilde H^n(A,E)&:=&\mathcal Z^n(A,E)/\widetilde N^n(A,E).
\end{eqnarray*}
\begin{dfn}
Let $A$ be a Banach algebra. A \emph{Banach extension of $A$ by $I$} is a short exact sequence of Banach algebras and bounded homomorphisms
$$\Sigma=\Sigma(\mathfrak A;I):0\rightarrow I\stackrel{\iota}{\rightarrow}\mathfrak{A}\stackrel{q}{\rightarrow}A\rightarrow 0 $$ 
where  $\iota(I)$ is a closed ideal in $\mathfrak A$. We say $\Sigma$ is \emph{admissible} if there exists a bounded linear map $Q:A\rightarrow\mathfrak A$ with 
$q\circ Q=\up{id}_{\mathfrak A}$ (i.e. if $\Sigma$ splits as a short exact sequence in the category of Banach spaces and bounded 
linear maps). A Banach extension \emph{splits strongly} if it has a splitting homomorphism that  is continuous.

Two Banach extensions, $\Sigma(\mathfrak A;I)$ and $\Sigma(\mathfrak B;I)$ of $A$ by $I$  are 
\emph{strongly equivalent} if there is a continuous 
algebra isomorphism, $\psi:\mathfrak A\rightarrow\mathfrak B$, such that $\psi(x)=x$ ($x\in I$) and the diagram (\ref{comdia}) commutes; 
it is standard that this defines an equivalence relation on the class of all Banach extensions of $A$ by $I$. We denote the 
equivalence class, with respect to strong equivalence, of a Banach extension $\Sigma$ 
by $[\Sigma]_s$.
\end{dfn}
If $\Sigma(\mathfrak A;I)$ is a Banach extension of $A$ and  $a,b\in \mathfrak A$ with $q(a)=q(b)$ then there exists $e\in I$ such that $a=b+e$. If in addition $\Sigma(\mathfrak A;I)$ is singular, then, for $x\in I$, we have
\begin{equation}\label{welldefined}
\iota(x)a=\iota(x)(b+\iota(e))=\iota(x)b,
\end{equation}
Hence, we may make $I$ into an $A$-bimodule by defining left and right actions as follows.
\[
x\cdot q(a)=\iota(x)a,\,q(a)\cdot x=a\iota(x),\qquad a\in\mathfrak A,x\in I
\]
If $E$ is a Banach $A$-bimodule and $\Sigma=\Sigma(\mathfrak A;I)$ is a singular Banach extension of $A$ 
such that $I$ with the above operations is isomorphic to $E$ as a Banach $A$-bimodule, we say $\Sigma$ is a 
\emph{singular Banach extension of $A$ by $E$}.
Now let $T\in\mathcal Z^2(A,E)$. We set
\begin{eqnarray*}
\mathfrak A_T=A\oplus_T E&=&\{(a,x):a\in A,\,x\in E\},\\
(a,x)(b,y)&=&\left(ab, a\cdot y+x\cdot b +T(a,b)\right),\,\left((a,x), (b,y)\in\mathfrak A_T\right)
\end{eqnarray*}
and define maps $\iota_T:E\rightarrow\mathfrak A_T, \,\iota_T(x)=(0,x)$ and $q_T:\mathfrak A_T\rightarrow A, \,q_T((a,x))=a$. It is then standard 
(see \cite[p.278]{Dales}) that $\mathfrak A_T$ is a Banach algebra with respect to a norm equivalent to the one given by 
\[
\norm{(a,x)}_1=\norm a + \norm x,\,((a,x)\in\mathfrak A_T),
\]
and that if we equip $\mathfrak A_T$ with this algebra norm  then 
\[
\Sigma(\mathfrak A_T;E):0\rightarrow E \stackrel{\iota_T}{\rightarrow}\mathfrak{A}\stackrel{q_T}{\rightarrow}A\rightarrow 0 
\] 
is a singular, admissible Banach extension of $A$ by $E$. We 
shall denote $\Sigma(\mathfrak A_T;E)$ by $\Sigma_T$. The following is a slightly rewritten version of \cite[2.8.12]{Dales}.
\begin{prop}\label{Dalesext}
Let $A$ be a Banach algebra, $E$ a Banach $A$-bimodule, and
$T, T'\in\mathcal Z^2(A,E)$. Then:
\begin{itemize} 
\item[(a)] if $T-T'\in \widetilde N^2(A,E)$, $\Sigma_T$ is equivalent to $\Sigma_{T'}$. Moreover, 
$$T+\widetilde N^2(A,E)\mapsto[\Sigma_T]$$ 
is a bijection from $\widetilde H^2(A,E)$ to the family of equivalence classes, with respect to equivalence, of singular, admissible Banach extensions of $A$ by $E$; 
\item[(b)] if $T-T'\in \mathcal N^2(A,E)$, $\Sigma_T$ is strongly equivalent to $\Sigma_{T'}$. Moreover, 
$$T+\mathcal N^2(A,E)\mapsto[\Sigma_T]_s$$
is a bijection from $\mathcal H^2(A,E)$ to the family of equivalence classes, with respect to strong equivalence, of singular, admissible Banach extensions of $A$ by $E$.
\end{itemize}
\end{prop}

\subsection{Compact cocycles}\index{cocycle!compact}
We now move on to some new definitions and results.
\begin{dfn}\label{groups}We define the following vector spaces as analogues of those in Subsection \ref{Kamsub}:
\begin{eqnarray*}
Z_K^n(A,E)&:=&\mathcal Z^n(A,E)\cap\mathcal K^n(A,E);\\ 
N_K^n(A,E)&:=&\mathcal N^n(A,E)\cap\mathcal K^n(A,E)(=\mathcal N^n(A,E)\cap Z^n_K(A,E));\\ 
\widetilde N_K^n(A,E)&:=& N^n(A,E)\cap\mathcal K^n(A,E)(=N^n(A,E)\cap Z^n_K(A,E));\\ 
H_K^n(A,E)&:=& Z_K^n(A,E)/N_K^n(A,E);\\ 
\widetilde H_K^n(A,E)&:=& Z_K^n(A,E)/\widetilde N_K^n(A,E).
\end{eqnarray*}
Similarly, we define the following weakly compact analogues of the groups in Subsection \ref{Kamsub}:
\begin{eqnarray*}
Z_w^n(A,E)&:=&\mathcal Z^n(A,E)\cap w\mathcal K^n(A,E);\\ 
N_w^n(A,E)&:=&\mathcal N^n(A,E)\cap w\mathcal K^n(A,E)(=\mathcal N^n(A,E)\cap Z_w^n(A,E));\\ 
\widetilde N_w^n(A,E)&:=& N^n(A,E)\cap w\mathcal K^n(A,E)(= N^n(A,E)\cap Z_w^n(A,E));\\ 
H_w^n(A,E)&:=& Z_w^n(A,E)/N_w^n(A,E);\\ 
\widetilde H_w^n(A,E)&:=& Z_w^n(A,E)/\widetilde N_w^n(A,E).
\end{eqnarray*}
\end{dfn}
Note that $H_K^1(A,E)=\widetilde H_K^1(A,E)$ and is zero if and only if all compact derivations from $A$ to $E$ are inner. If $A$ is a 
commutative Banach algebra and $E$ is a symmetric Banach $A$-bimodule, this is equivalent to there being no non-zero, compact derivations from 
$A$ to $E$. Similarly, $H_w^1(A,E)=\widetilde H_w^1(A,E)$ is zero if and only if all weakly  compact derivations from $A$ to $E$ are inner and, 
if $A$ is a commutative Banach algebra and $E$ a symmetric Banach $A$-bimodule, this is equivalent there being no non-zero 
weakly compact derivations from $A$ to $E$.
\begin{definition}
We say that a Banach extension 
$$\Sigma:0\rightarrow I\stackrel{\iota}{\rightarrow}\mathfrak{A}\stackrel{q}{\rightarrow}A\rightarrow 0 $$ 
\emph{cf-splits} if there is a cf-homomorphism $Q$ such that $q\circ Q=\up{id}_A$. We call $Q:A\rightarrow \mathfrak A$ a \emph{splitting cf-homomorphism}. We say 
$\Sigma$ 
\emph{wcf-splits} if there is a wcf-homomorphism $Q:A\rightarrow \mathfrak A$ such that $q\circ Q=I_a$. In this case we call 
$Q$ a \emph{splitting wcf-homomorphism}. If the extension $\Sigma$ cf-splits [respectively wcf-splits] with a bounded splitting cf-homomorphism  [respectively wcf-homomorphism], we say that $\Sigma$ \emph{cf-splits strongly} [respectively \emph{wcf-splits strongly}].
\end{definition}

 Note that ``(w)cf-splitting strongly'' can be thought of as ``splitting in the category of Banach algebras and 
bounded (w)cf-homomorphisms''. The  statements in the following lemma that refer to splitting and splitting strongly  are well known but do 
not seem to be explicitly stated in the standard textbooks. The statements relating to (w)cf-splitting are new. The proofs 
are trivial and are omitted.
\begin{lem}\label{equivsplit}
Let $A$ be an algebra and let $\Sigma(\mathfrak A;I)$ and $\Sigma(\mathfrak B;I)$ be equivalent extensions of 
$A$. If $\Sigma(\mathfrak A;I)$ splits, then $\Sigma(\mathfrak B;I)$ splits.

Let $A$ be a Banach algebra and let $\Sigma(\mathfrak A;I)$ and $\Sigma(\mathfrak B;I)$ be strongly equivalent Banach extensions of 
$A$. Then the following hold:
\begin{itemize}
\item if $\Sigma(\mathfrak A;I)$ splits strongly, then $\Sigma(\mathfrak B;I)$ splits strongly;
\item if $\Sigma(\mathfrak A;I)$ cf-splits [respectively cf-splits strongly], then $\Sigma(\mathfrak B;I)$ cf-splits [respectively cf-splits strongly];
\item if $\Sigma(\mathfrak A;I)$ wcf-splits [respectively wcf-splits strongly], then $\Sigma(\mathfrak B;I)$ wcf-splits [respectively wcf-splits strongly].
\end{itemize}
\end{lem}
For the remainder of this section we will refer only to the norm topology case. In all cases the weak topology version of any result holds and the proof is basically identical.

The following is an analogue of \ref{Dalesext} in this new setting.
\begin{thm}\label{Kext}
Let $A$ be a Banach algebra, $E$ a Banach $A$-bimodule, and
$T, T'\in Z_K^2(A,E)$. Then:
\begin{itemize} 
\item[(a)] $\Sigma_T$ cf-splits strongly. 
\item[(b)]if $T-T'\in \widetilde N_K^2(A,E)$, $\Sigma_T$ is equivalent to $\Sigma_{T'}$. Moreover,
$$T+\widetilde N_K^2(A,E)\mapsto[\Sigma_T]$$ 
is a bijection from $\widetilde H_K^2(A,E)$ to the family of equivalence classes $C$, with respect to equivalence, of singular, admissible Banach extensions of $A$ by $E$ such that $C$ contains an extension that cf-splits (or equivalently contains an extension that cf-splits strongly); 
\item[(c)]if $T-T'\in N_K^2(A,E)$, $\Sigma_T$ is strongly  equivalent to $\Sigma_{T'}$. Moreover,
$$T+ N_K^2(A,E)\mapsto[\Sigma_T]_s$$
is a bijection from $H_K^2(A,E)$ to the family of equivalence classes, with respect to strong equivalence, of singular, admissible
Banach extensions of $A$ by $E$ that cf-split strongly.
\end{itemize}
\end{thm}
\begin{proof}
First, we prove part (a). Let $T\in Z_K^2(A,E)$ be arbitrary and define a bounded linear map $Q:A\rightarrow \mathfrak A_T$ by $Q(a)=(a,0)$, $(a\in A)$. Then, for $a,b\in A$,
\begin{eqnarray*}
S_Q(a,b):&=&Q(a)Q(b)-Q(ab)=(a,0)(b,0)-(ab,0)\\
&=&(ab,T(a,b))-(ab,0)=(0,T(a,b)),
\end{eqnarray*}
and so $S_Q$ is compact. Clearly, $Q\circ q=\up{id}_A$ so $\Sigma_T$ cf-splits strongly  with bounded splitting cf-homomorphism $Q$ and so part (a) holds.

Now we prove parts (b) and (c).
Let $T, T'\in Z_K^2(A,E)$.
Suppose that $T-T'\in \widetilde N_K^2(A,E)$. Then $T-T'\in \widetilde N^2(A,E)$, and so $\Sigma_T$ is equivalent to $\Sigma_{T'}$ 
by part (a) of Proposition \ref{Dalesext}. 

Suppose further that $T-T'\in N_K^2(A,E)$. Then $T-T'\in \mathcal N^2(A,E)$, and so $\Sigma_T$ is strongly equivalent to 
$\Sigma_{T'}$ by part (a) of Proposition \ref{Dalesext}. 

Now suppose instead that $[\Sigma_T]=[\Sigma_{T'}]$. Then $T-T'\in \widetilde N^2(A,E)$ by part (a) of Proposition \ref{Dalesext}. Also, $T-T'\in Z_K^2(A,E)$ by assumption so $T-T'\in \widetilde N_K^2(A,E)$ and $T+\widetilde N_K^2(A,E)\mapsto[\Sigma_T]$ is injective.

Suppose further that $[\Sigma_T]_s=[\Sigma_{T'}]_s$. Then $T-T'\in \mathcal N^2(A,E)$ by part (a) of Proposition \ref{Dalesext}. Also, $T-T'\in Z_K^2(A,E)$ by assumption so $T-T'\in N_K^2(A,E)$ and $T+N_K^2(A,E)\mapsto[\Sigma_T]_s$ is injective.

That the two maps are into the collection of equivalence classes (with respect to the relevant relation) of extensions that cf-split follows from part (a), proven above.

It only remains to show that the maps are surjective, i.e. that, for each singular Banach extension $\Sigma$ of $A$ by $E$ that cf-splits, there 
exists $T\in Z_K^2(A,E)$ with $\Sigma_T$ equivalent to $\Sigma$ and that if $\Sigma$ cf-splits strongly we may take $\Sigma_T$ to be strongly equivalent to $\Sigma$. Let 
$$\Sigma=\Sigma(\mathfrak{A};E):0\rightarrow E\stackrel{\iota}{\rightarrow}\mathfrak{A}\stackrel{q}{\rightarrow}A\rightarrow 0$$
be a singular Banach extension of $A$ by $E$ with splitting cf-homomorphism $Q$. By the definition of a cf-homomorphism, we have that $S_Q\in\mathcal K^2(A,\mathfrak A)$. Now, since $q$ is a homomorphism,
\begin{eqnarray*}
q\circ S_Q(a,b)&=&q(Q(a)Q(b)-Q(ab))\\
&=& q\circ Q(a)q\circ Q(b)-q\circ Q(ab)\\
&=&ab-ab=0, \quad(a,b\in A).
\end{eqnarray*}
 Hence, $S_Q(A^{(2)})\subseteq \up{ker}(q)=E$ and so we can define $T\in\mathcal K^2(A,E)$ by $T(a,b)=S_Q(a,b)$, $(a,b\in A)$. Furthermore, a direct calculation yields,
\[
\delta^2(T)(a,b,c)=0
\]
Hence, $T\in Z_K^2(A,E)$. We claim that the Banach extension $\Sigma_T$ is  equivalent to $\Sigma$. For $a\in\mathfrak A$, 
$a-Q(q(a))\in E$ so we may define a map
\func{\psi}{\mathfrak A}{\mathfrak A_T}{a}{(q(a),a-Q(q(a)))} 
It is clear that $\psi$ is linear. Furthermore, if $Q$ is bounded (which we may assume if $\Sigma$ cf-splits strongly) then $\psi$ is also bounded. Also, if we define a map $\phi:\mathfrak A_T\rightarrow\mathfrak A$ by 
$\phi((b,e))=Q(b)+e$ it is easily checked that $\phi$ and $\psi$ are mutually inverse. Further, $q_T\circ\psi(a)=q(a)$, 
$(a\in \mathfrak A)$ so $\psi$ makes the diagram (\ref{comdia}) commute. It remains only to show that $\psi$ is an algebra homomorphism. 
Let $a,b\in\mathfrak A$; then a direct calculation yields
\[
\psi(a)\psi(b)=\big(q(ab),ab-Q(q(ab))+(a-Q(q(a)))(b-Q(q(b)))\big),
\]
but $a-Q(q(a)),b-Q(q(b))\in E$ so $(a-Q(q(a)))(b-Q(q(b)))=0$ and so
\[
\psi(a)\psi(b)=(q(ab),ab-Q(q(ab)))=\psi(ab).
\]
Thus the result holds.
\end{proof}
Note that part (b) of the above theorem implies that, if a singular, admissible Banach extension of $A$ by $E$ cf-splits, then it is equivalent to a singular, admissible Banach extension of $A$ by $E$ that cf-splits strongly.

This gives us the following corollaries.
\begin{corollary}\label{Kextcor}
Let $A$ be a Banach algebra, and let $E$ be a Banach $A$-bimodule.
\begin{enumerate}
\item The following are equivalent:
\begin{itemize}
\item[(a)] $\widetilde H_K^2(A,E)=\{0\}$;
\item[(b)] each singular Banach extension of $A$ by $E$, which cf-splits, does split.
\end{itemize}
\item The following are equivalent:
\begin{itemize}
\item[(a)] $H_K^2(A,E)=\{0\}$;
\item[(b)] each singular Banach extension of $A$ by $E$, which cf-splits strongly, does split strongly.
\end{itemize}
\end{enumerate}
\end{corollary}
\begin{proof}
\begin{enumerate}
\item
To show that (a) implies (b), let  $\widetilde H_K^2(A,E)=\{0\}$ and let 
\[
\Sigma:0\rightarrow I\stackrel{\iota}{\rightarrow}\mathfrak{A}\stackrel{q}{\rightarrow}A\rightarrow 0 
\]
be a Banach extension of $A$ by $E$ which cf-splits. By Theorem \ref{Kext}, $\Sigma$ is equivalent to the Banach extension $\Sigma_0$ (that is the Banach extension $\Sigma_T$ where $T$ is the zero map):
\[
\Sigma_0:0\rightarrow I\stackrel{\iota_0}{\rightarrow}\mathfrak{A}\stackrel{q_0}{\rightarrow}A\rightarrow 0 
\]
where, for $x\in E$ and $a\in A$,   $\iota_0(x)=(0,x)$ and $q_0((a,x))=a$. The extension 
$\Sigma_0$ splits strongly with the continuous splitting homomorphism $\theta:A\rightarrow\mathfrak A_0$, 
$\theta(a)=(a,0)$, $(a\in A)$. By Lemma \ref{equivsplit}, $\Sigma$ splits. 

To show that (b) implies (a), we assume that each singular Banach 
extension of $A$ by $E$, which cf-splits, 
splits and let $T\in Z_K^2(A,E)$. Then, by Theorem \ref{Kext},  $\Sigma_T$ splits; let $\theta$ be a splitting homomorphism for $\Sigma_T$. Since, 
$q_T\circ\theta=\up{id}_A$ it follows that there exists $S\in \mathcal L(A,E)$ with $\theta(a)=(a,S(a))$ $(a\in A)$. Hence,
\begin{eqnarray*}
(ab,S(ab))&=&\theta(ab)=\theta(a)\theta(b)\\
&=&(a,S(a))(b,S(b))=(ab,a\cdot S(b)+S(a)\cdot b+T(a,b)),
\end{eqnarray*}
 so
\[
S(ab)=a\cdot S(b)+S(a)\cdot b+T(a,b),
\]
and 
\[
T(a,b)=a\cdot(-S(b))+(-S(a))\cdot b -(-S(ab))=\delta^1(-S)(a,b).
\]
Thus $T=\delta^1(-S)\in N^2(A,E)$ and so $\widetilde H_K^2(A,E)=\{0\}$.

\item To show (a) implies (b) let $H_K^2(A,E)=\{0\}$ and $\Sigma$ be a singular Banach extension of $A$ by $E$, which cf-splits strongly. By Theorem 
\ref{Kext}, $\Sigma$ is strongly equivalent to $\Sigma_0$ and so, by Lemma \ref{equivsplit}, $\Sigma$ splits strongly. 

To show that (b) implies (a), assume that each singular Banach extension of $A$ by $E$, which 
cf-splits strongly, splits strongly and let $T\in Z_K^2(A,E)$. Then $\Sigma_T$ splits strongly; let $\theta$ be a continuous splitting 
homomorphism for $\Sigma_T$. Since, $q_T\circ\theta=\up{id}_A$ it follows that there exists $S\in \mathcal B(A,E)$ with 
$\theta(a)=(a,S(a))$ $(a\in A)$. As in the proof of the first part of this result, $T=\delta^1(-S)$ and so $T\in\mathcal N^2(A,E)$. 
Hence $H_K^2(A,E)=\{0\}$.
\end{enumerate}
\end{proof}
The following result gives us a new way of showing that bounded cf-homomorphisms need not be homomorphisms or compact linear maps (which we showed directly in Example \ref{nhex}). 
\begin{cor}\label{nhcor}
Let $A$ be an infinite-dimensional Banach algebra and $E$ a Banach $A$-bimodule such that $H^2_K(A,E)\ne\{0\}$, then there exists a Banach algebra 
$\mathfrak A$ with underlying Banach space isomorphic to $A\oplus_1 E$ and a bounded cf-homomorphism $Q:A\rightarrow\mathfrak A$ which is 
neither a homomorphism nor a compact linear map.
\end{cor}
\begin{proof}
 By Corollary \ref{Kextcor} there exists a Banach extension,
\[
\Sigma:0\rightarrow E\stackrel{\iota}{\rightarrow}\mathfrak{A}\stackrel{q}{\rightarrow}A\rightarrow 0,
\]
of $A$ by $E$ which does not split strongly but such that there is a bounded cf-homomorphism, 
$Q:A\rightarrow \mathfrak A$ with $q\circ Q=\up{id}_A$. Since $\Sigma$ does not split strongly, $Q$ is not a homomorphism, and since 
$q\circ Q=\up{id}_A$ is not a compact linear map, neither is $Q$. By Theorem \ref{Kext}, $\Sigma$ is strongly equivalent to $\Sigma_T$ 
for some $T\in N^2_K(A,E)$ and so $\mathfrak A_T$ is isomorphic as a Banach space to  $A\oplus_1 E$.
\end{proof}
Below is an example of a choice of $A$ and $E(=A^*)$ satisfying the hypotheses of Corollary \ref{nhcor}. The author would like to thank Yemon Choi for sending him some notes, which help with this construction.
\begin{example}
Let $A$ be $\ell^2$ equipped with the pointwise product. A direct calculation gives that the map 
$\gamma^1: \mathcal B(A, A^*)\rightarrow  \mathcal B^2(A, A^*)$ is injective.
For $N\in\N$ we set $G_N:A\rightarrow A^*$ to be the bounded linear map given by 
\[G_N(e_k)(e_j)=g_{N,j,k}=
\left\{
\begin{array}{ll}
\frac{1}{3\sqrt{2N+1}}&\textrm{if }\abs j,\abs k\le N\\
0&\textrm{otherwise}
\end{array}
\right.\]
and set $F_N:=\gamma^1(G_N)$. We have that $F_N$ is finite rank and bounded (an thus that it is a compact bilinear map).

A direct calculation gives that, for each $N\in\N$, $\norm{F_N}\le1$ and $\norm{G_N}\rightarrow\infty$ as $N\rightarrow \infty$. It follows from an application of the Banach isomorphism theorem that $N_K^2(A;A^*)$ cannot be complete. In particular $N_K^2(A;A^*)\ne Z_K^2(A;A^*)$ i.e. $H_K^2(A,A^*)\ne0$.
\end{example}
\section{Open questions}
We finish by listing some questions relating to the material in this paper.
\begin{question}
For well known examples of Banach algebras, $A$, what are the automorphisms of $A$ in the category of Banach algebras with cf-homomorphisms; what are the automorphisms of $A$ in the category of Banach algebras with wcf-homomorphisms?
\end{question}
\begin{question}
Does there exist a wcf-homomorphism which is neither a weakly compact linear map nor a cf-homomorphism? 
\end{question}
\begin{question}
What can we say about (weakly) compact failure of other algebraic identities: for example, commutativity?
\end{question}
\begin{question}
What do the groups $H_K^n(A,E)$ (and the others introduced in Definition \ref{groups}) tell us about $A$ and $E$ when $n>2$?
\end{question}

\bibliographystyle{plain} 
\def\cprime{$'$}

\noindent M. J. Heath\\
Departamento de Matem\'atica \\
Instituto Superior T\'ecnico \\
Av. Rovisco Pais \\
1049-001 Lisboa \\
Portugal \\
mheath@math.ist.utl.pt

\end{document}